\documentclass[11pt]{article}
\usepackage[english]{babel}
\usepackage{amsmath,amsthm,amssymb}
\usepackage[all]{xy}
\usepackage{setspace}
\usepackage[top=1.3in, bottom=1.6in, left=1.3in, right=1.3in]{geometry}
\usepackage{color}
\definecolor{grau}{rgb}{0.5,0.5,0.5}
\usepackage[colorlinks, linkcolor=grau, citecolor=grau, urlcolor=grau]{hyperref}
\usepackage{breakurl}
\usepackage{bibspacing}
\usepackage{leftidx}
\newtheorem{theorem}{Theorem}[section]
\newtheorem{lemma}[theorem]{Lemma}
\newtheorem{proposition}[theorem]{Proposition}

\numberwithin{equation}{section}

\def\sp{\textnormal{Spec}}

\def\P1k{\mathbb P_{k}^{1}}
\def\deg{\textnormal{deg}}

\def\sp{\textnormal{Spec}}

\newcommand {\Q} {{\mathbb Q}}

\newcommand {\QQ}  {{\mathbb Q}}
\newcommand {\ZZ}  {{\mathbb Z}}

\newcommand {\CC}  {{\mathbb C}}

\newcommand {\cc} {sp} 
\newcommand {\bb} {(sp)^*}
\newcommand {\p} {x}
\newcommand {\br} {\lambda}
\newcommand {\Br} {\Lambda}
\newcommand {\bk} {\bar{K}}
\newcommand {\uu} {m}

\date{}

\begin{document}

\author{Ariyan Javanpeykar and Rafael von K\"anel}
\title{Szpiro's small points conjecture for cyclic covers}
\maketitle

\begin{abstract}
\small Let $X$ be a smooth, projective and geometrically connected curve of genus at least two, defined over a number field. In 1984, Szpiro conjectured that $X$  has a  ``small point''.  
In this paper we prove that if $X$ is a cyclic cover of prime degree of the projective line, then $X$ has infinitely many ``small points''. In particular, we establish the first cases of Szpiro's small points conjecture, including the genus two case and the hyperelliptic case. The proofs use Arakelov theory for arithmetic surfaces and the theory of logarithmic forms.
\end{abstract}

\section{Introduction}\label{sec:intro}

Let $X$ be a smooth, projective and geometrically connected curve of genus at least two, defined over a number field. In 1984, Szpiro \cite{szpiro:faltings} conjectured that  $X$ has a ``small point'', where a ``small point'' is an algebraic point of the curve $X$ with ``height'' bounded from above in a certain way. We refer to Section \ref{sec:conj} for a precise formulation of Szpiro's small points conjecture.
Szpiro  proved that his conjecture implies  an ``effective Mordell conjecture''. 
In other words, Szpiro's remarkable approach shows that to bound the height of all rational points of any curve $X$, it suffices to produce for any curve $X$ at least one ``small point''.
The small points conjecture was studied in Szpiro's influential seminars \cite{szpiro:spa,szpiro:spe}, see also Szpiro's articles \cite{szpiro:lefschetz,szpiro:grothendieck}.

The results of this paper are as follows. Let $\mathcal C$ be the set of curves $X$ as above which are cyclic covers of prime degree of the projective line. 
For example, if $X$ has genus two or is hyperelliptic, then $X\in\mathcal C$.
Our first result (see Theorem \ref{thm:cyclic}) gives that any $X\in\mathcal C$ has infinitely many ``small points''.  
In particular, we establish the first cases of  Szpiro's small points conjecture.
Furthermore, we improve Theorem \ref{thm:cyclic} for hyperelliptic curves (see Theorem \ref{thm:hyper}) and for genus two curves (see Theorem \ref{thm:genus2}) in the sense that we produce ``smaller points'' on such curves. 
To give the reader a more concrete idea of our results we now state a special case of Theorem \ref{thm:genus2}. If the curve $X$ has genus two and is defined over $\QQ$, then $X$ has infinitely many algebraic points $\p$ that satisfy
$$\max\bigl(h_{NT}(\p),h(\p)\bigl)\leq (10\prod p)^{10^6}$$
with the product taken over all bad reduction primes $p$ of $X$. Here $h_{NT}$ is the N\'eron-Tate height and $h$ is the stable Arakelov height, see Section \ref{sec:conj}. 
We also give in Proposition \ref{prop:abc} and Proposition \ref{prop:cyclic} versions of the above theorems with ``exponentially smaller points''. 
These versions are either conditional on the $abc$ conjecture, or they depend on lower bounds for Faltings' delta invariant  which are not known to be effective.

We remark that ``effective Siegel or Mordell'' applications of our completely explicit results require hyperbolic curves which admit Kodaira-Par{\v{s}}in type constructions with all fibers in $\mathcal C$. There exists no hyperbolic curve for which  Kodaira's construction (see for example \cite[p.266]{szpiro:spa})  is of this form and we are not aware of a hyperbolic curve for which Par{\v{s}}in's construction (see Par{\v{s}}in \cite{parshin:construction}, or \cite[p.268]{szpiro:spa}) has all fibers in $\mathcal C$. Therefore our results have at the moment no ``effective Siegel or Mordell'' applications. However, there is the hope that new and more suitable Kodaira-Par{\v{s}}in type constructions will be discovered. For example, recently Levin \cite{levin:siegelshaf} gave a new Par{\v{s}}in type construction for integral points on hyperelliptic curves with all fibers in $\mathcal  C$. 
 
The main ingredients for our proofs are as follows. Let $X$ be a smooth, projective and geometrically connected curve of genus $g\geq 2$, defined over a number field $K$. On using fundamental results of Arakelov \cite{arakelov:intersectiontheory}, Faltings \cite{faltings:arithmeticsurfaces} and Zhang \cite{zhang:positive}, we show that for any $\epsilon>0$ there exist infinitely many algebraic points $\p$ of $X$ that satisfy 
\begin{equation}\label{eq:sp}
h_{NT}(\p)\leq  2g(g-1)h(\p)\leq  ge(X)+\epsilon.
\end{equation}
Here $e(X)$ is a certain stable Arakelov self-intersection number, defined in (\ref{def:ex}). Thus, to produce ``small points'' of $X$, it suffices to estimate $e(X)$ effectively in terms of $K,S$ and $g$, for $S$ the set of finite places of $K$ where $X$ has bad reduction.

In the proof of Theorem \ref{thm:cyclic} we use properties of the Belyi degree $\deg_B(X)$ of $X$, which is defined in (\ref{def:belyi}). From \cite[Theorem 1.1.1]{javanpeykar:belyi} we obtain
\begin{equation}\label{eq:ex1}
e(X)\leq 10^8\deg_B(X)^5g.
\end{equation}
To control $\deg_B(X)$ we use an effective version  of Belyi's theorem \cite{belyi:theorem}, worked out by Khadjavi in \cite{khadjavi:belyi}. We deduce an explicit upper bound for $\deg_B(X)$ in terms of $K,g$ and the heights $h(\br)$ of the cross-ratios $\br$ of four (geometric) branch points of a finite morphism $X\to\mathbb P^1_K$. 
To estimate $h(\br)$ we assume that $X\in\mathcal C$ and then we combine  \cite[Proposition 2.1]{jore:shafarevich} of de Jong-R\'emond with \cite[Proposition 6.1 (ii)]{rvk:szpiro}. 
The latter result is based on the theory of logarithmic forms and the former result generalizes ideas of Par{\v{s}}in \cite{parshin:shafarevich} and Oort \cite{oort:shafarevich}. 
This leads to an explicit upper bound, in terms of $K,S$ and $g$, for $\deg_B(X)$, then for $e(X)$ by (\ref{eq:ex1}), and then finally for $h_{NT}(x)$ and $h(x)$ by (\ref{eq:sp}). 

To discuss the proof of Theorem \ref{thm:genus2} we assume that $g=2$. From \cite[Proposition 5.1 (v)]{rvk:szpiro} we obtain that Faltings' delta invariant $\delta(X_\CC)$ of a compact connected Riemann surface $X_\CC$ of genus two satisfies $-186\leq \delta(X_\CC).$ Then the Noether formula  for arithmetic surfaces, due to Faltings \cite{faltings:arithmeticsurfaces}  and Moret-Bailly \cite{moba:noether}, leads to 
\begin{equation}\label{eq:ex2}
e(X)\leq 12h_F(J)+201.
\end{equation}
Here $h_F(J)$ is the absolute stable Faltings height of the Jacobian $J=\textnormal{Pic}^0(X)$ of $X$. 
The inequalities in \cite[Proposition 4.1 (i), Proposition 6.1 (ii)]{rvk:szpiro} slightly refine a method developed by Par{\v{s}}in \cite{parshin:shafarevich}, Oort \cite{oort:shafarevich} and de Jong-R\'emond \cite{jore:shafarevich}. On combining these inequalities, we deduce an explicit upper bound, in terms of $K,S$ and $g$, for $h_F(J)$, then for $e(X)$ by (\ref{eq:ex2}), and then finally for $h_{NT}(x)$ and $h(x)$ by (\ref{eq:sp}). 

To prove Theorem \ref{thm:hyper} we replace in the proof of Theorem \ref{thm:genus2} the results from Arakelov theory by a formula of de Jong \cite[Theorem 4.3]{dejong:weierstrasspoints} and the explicit estimate  for the hyperelliptic discriminant modular form which is given in \cite[Lemma 5.4]{rvk:szpiro}. 

The plan of this paper is as follows. In Section \ref{sec:conj} we discuss Szpiro's small points conjecture and its variation which involves the N\'eron-Tate height. 
In Section \ref{sec:results} we state our theorems, and we give Proposition \ref{prop:abc} which is conditional on the $abc$ conjecture. 
Then in Section \ref{sec:curves} we collect results from Arakelov theory for arithmetic surfaces. We also give an upper bound for the Belyi degree of $X$. 
In Section \ref{sec:cyclic} we consider curves $X\in \mathcal C$ and we prove Theorem \ref{thm:cyclic}. We also establish Proposition \ref{prop:cyclic}. It refines our theorems, with the disadvantage inherent that it involves a lower bound for Faltings' delta invariant which is not known to be effective. 
In Section \ref{sec:hyper} we first show  Lemma \ref{lem:tx} and Lemma \ref{lem:wpf}. They give explicit results for certain (Arakelov) invariants of hyperelliptic curves. Then we use these lemmas to prove 
Theorem \ref{thm:hyper}. Finally, in Section \ref{sec:genus2}, we give a proof of Theorem \ref{thm:genus2}.

Throughout this paper we shall use the following notations and conventions. Let $K$ be a number field. We denote by $\bar{K}$ a fixed algebraic closure of $K$. If $L$ is a field extension of $K$, then we write $[L:K]$ for the relative degree of $L$ over $K$. By a curve $X$ over $K$ we mean a smooth, projective and geometrically connected curve $X\to\sp(K)$. 
For any finite place $v$ of $K$, we write $N_v$ for the number of elements in the residue field of $v$ and we let $v(\mathfrak a)$ be the order of $v$ in a fractional ideal $\mathfrak a$ of $K$. We denote by $\lvert S\rvert$ the cardinality of a set $S$. Finally, by $\log$ we mean the principal value of the natural logarithm and we define the product taken over the empty
set as $1$.

\section{Small points conjectures}\label{sec:conj}

In this section we state and discuss Szpiro's small points conjecture, and its variation which involves the N\'eron-Tate height. Let $K$ be a number field and let $g\geq 2$ be an integer. We denote by $X$ a curve over $K$ of genus $g$. 

We take an algebraic point $\p\in X(\bar{K})$. 
The classical result of Deligne-Mumford gives a finite extension $L$ of $K$ such that $X_L$ has semi-stable reduction over $B$ and such that $\p\in X(L)$, where $B$ is the spectrum of the ring of integers of $L$.  
We denote by $\omega_{\mathcal X/B}$ the relative dualizing sheaf of the minimal regular model $\mathcal X$ of $X_L$ over $B$.  Let $\omega=(\omega_{\mathcal X/B}, \lVert \cdot \rVert)$  with $\lVert \cdot \rVert$ the Arakelov metric, let $(\cdot ,\cdot)$ be the intersection product of Arakelov divisors on $\mathcal X$ and identify $\omega$ and $\p$  with the corresponding Arakelov divisors on $\mathcal X$, see \cite[Section 2]{faltings:arithmeticsurfaces} for definitions. 
Then the stable Arakelov height $h(\p)$ of $\p$ is the real number defined by
\begin{equation}\label{def:h}
[L:\Q]h(\p)=(\omega,\p).
\end{equation} 
The factor $[L:\QQ]$ and the semi-stability of $\mathcal X$ assure that the definition of $h(x)$ does not depend on the choice of a field  $L$ with the above properties. 

Let $S$ be a set of finite places of $K$. We say that a constant $c$, depending on the data $(\mathcal D)$, is effective if one can in principle explicitly compute the real number $c$ provided that $(\mathcal D)$ is given. In 1984, Szpiro \cite[p.101]{szpiro:faltings} formulated his small points conjecture in terms of the Arakelov self-intersection $(\p,\p)$ of $x$. However, Arakelov's  adjunction formula  $(\omega,\p)=-(\p,\p)$ in \cite{arakelov:intersectiontheory}
shows that Szpiro's small points conjecture coincides with:

\vspace{0.3cm}
\noindent{\bf Conjecture $(\cc)$.}
\emph{There exists an effective constant $c$, depending only on $K$, $S$ and $g$, with the following property. Suppose $X$ is a curve over $K$ of genus $g$,  with set of bad reduction places $S$. If $X$ has  semi-stable reduction over  the ring of integers of $K$, then there is a point $\p\in X(\bar{K})$ that satisfies $h(\p)\leq c.$
\vspace{0.3cm}}

It is known (see Szpiro \cite[p.101]{szpiro:faltings})  that this conjecture implies an ``effective Mordell conjecture'', and that Szpiro established a rather strong geometric analogue of Conjecture ($\cc$)  in \cite{szpiro:seminaire81}. 
We also mention that if Conjecture ($\cc$) holds, then it holds without the semi-stable assumption.  Indeed, on combining results of Grothendieck-Raynaud \cite[Proposition 4.7]{grra:neronmodels} and Serre-Tate \cite[Theorem 1]{seta:goodreduction} with Dedekind's discriminant theorem, one obtains a finite field extension $M$ of $K$ such that $X_M$ has semi-stable reduction over the ring of integers of $M$ and such that $[M:K]$ and the relative discriminant of $M$ over $K$ are effectively controlled in terms of $K$, $S$ and $g$. 

In the Grothendieck Festschrift, Szpiro \cite{szpiro:grothendieck} formulated   another version of Conjecture $(\cc)$. This formulation involves the N\'eron-Tate height
\begin{equation}\label{def:nt}
h_{NT}(\p)
\end{equation} 
of $\p\in X(\bar{K})$ which is defined as the N\'eron-Tate height of the divisor class $(2g-2)\p-\Omega^1$ in the Jacobian $\textnormal{Pic}^0(X_L)$ of $X_L$ for $L$ as above.
Here $\Omega^1$ denotes the sheaf of differential one-forms of the curve $X_L$ over $L$. We now give a version of Szpiro's ``Conjecture des deux petits points pour N\'eron-Tate'' which is stated in \cite[p.244]{szpiro:grothendieck}. 

\vspace{0.3cm}
\noindent{\bf Conjecture $\bb$.}
\emph{Any curve $X$ over $K$ of genus $g$ has two distinct points $\p_i\in X(\bar{K})$, $i=1,2$, that satisfy $h_{NT}(\p_i)\leq c,$ where $c$ is an effective constant which depends only on $K$, $g$ and the geometry of the bad reduction of $X$.\vspace{0.3cm}}

We point out that the conjecture in \cite[p.244]{szpiro:grothendieck} describes the constant $c$ in Conjecture $\bb$ more precisely. Szpiro conjectures in addition the existence of constants $a(g)$, $b(g)$, depending only on $g$, and constants $A(X)$, $B(X)$, depending only on the geometry of the bad reduction of the curve $X$, such that
\begin{equation*}
c=a(g)A(X)\log D_K+b(g)B(X).
\end{equation*}
Here $D_K$ denotes the absolute value of the discriminant of $K$ over $\QQ$. Par{\v{s}}in \cite{parshin:szpiro} showed that Conjecture $\bb$, with $c$ of the form as displayed above, implies Szpiro's classical discriminant conjecture for elliptic curves. In addition, we mention that Moret-Bailly proposed a  version of the small points conjecture which is equivalent to his adaption \cite[Hypoth\`ese BM]{moret-bailly:effmordell} of Par{\v{s}}in's conjecture in \cite{parshin:szpiro}, see \cite[Corollaire 3.5]{moret-bailly:effmordell}. The conjectures, which are considered in  \cite{parshin:szpiro} and \cite{moret-bailly:effmordell}, are inspired by the classical Bogomolov-Miyaoka-Yau inequality for algebraic surfaces.

We shall see in Lemma \ref{lem:conj} (i) that Conjecture ($\cc$) implies Conjecture $\bb$. An unconditional proof of the converse implication seems to be difficult.
However,  Szpiro's arguments in \cite[p.244]{szpiro:grothendieck} combined with Moret-Bailly's proof of \cite[Th\'eor\`eme 5.1]{moret-bailly:effmordell} show that Conjecture $\bb$ still implies an ``effective Mordell conjecture''.

Next, we discuss a possible dependence of the constants in Conjecture ($\cc$) and Conjecture $\bb$ on the important quantities $D_K$ and $N_S=\prod_{v\in S}N_v$.  We denote by $h_F(A)$  the absolute stable Faltings height of an abelian variety $A$ over $\QQ$ defined in \cite[p.354]{faltings:finiteness} and we denote by $N_A$ the norm of the usual conductor ideal of $A$. Now, we can state a version of Frey's height conjecture \cite[p.39]{frey:linksulm}. This version is inspired by its geometric analogue which is proven, for example, in Deligne \cite[p.14]{deligne:monodromie}.

\vspace{0.3cm}
\noindent{\bf Conjecture $(h)$.}
\emph{For any real $\epsilon>0$ and any integer $n\geq 1$, there is a constant $c$, depending only on $\epsilon,n$, such that if  $A$ is an abelian variety over $\QQ$ of dimension $n$, then $h_F(A)\leq (\frac{n}{2}+\epsilon)\log N_A+c.$\vspace{0.3cm}}

Lemma \ref{lem:conj} (ii) below shows that if Conjecture $(h)$ holds for some fixed $\epsilon>0$ and $n=[K:\QQ]g$, then there exists a constant $c'$, depending only on $[K:\QQ],g$ and the fixed $\epsilon$, with the following property. 
Any curve $X$ over $K$ of genus $g$, with set of bad reduction places $S$, has an algebraic point $x\in X(\bar{K})$ that satisfies $$ h(x)\leq \frac{3g}{g-1}\bigl(g+\frac{2\epsilon}{[K:\QQ]}\bigl)(\log N_S+\log D_K)+c'$$
and has two distinct points $x_i\in X(\bar{K})$, $i=1,2$, which both satisfy
$$ h_{NT}(x_i)\leq 12g^2\bigl(g+\frac{2\epsilon}{[K:\QQ]}\bigl)(\log N_S+\log D_K)+c'.$$
This conditional result suggests a possible dependence  of the constants in Conjecture $(\cc)$ and  $\bb$ on $D_K$ and $N_S$. To conclude our discussion we mention that the rather strong geometric version of Conjecture $(sp)$, established by Szpiro (see \cite[p.101]{szpiro:faltings}), suggests even sharper upper bounds for $h(x)$ and $h_{NT}(x_i)$ in terms of $D_K$ and $N_S$.

\section{Statements of results}\label{sec:results}
In this section we state and discuss  Theorem \ref{thm:cyclic}, Theorem \ref{thm:hyper} and Theorem \ref{thm:genus2}. We also give Proposition \ref{prop:abc} which is conditional on the $abc$ conjecture.

To state our results we need to introduce some notation. Let $g\geq 2$ be an integer, let $K$ be a number field and 
let $S$ be a set of finite places of $K$. We denote by $D_K$  the absolute value of the discriminant of $K$ over $\QQ$ and we denote by $d=[K:\QQ]$  the degree of $K$ over $\QQ$. 
To measure $K$, $S$ and $g$ we use the quantities
\begin{equation}\label{def:para}
D_K, \ d, \ N_S=\prod_{v\in S} N_v, \ g \ \textnormal{ and }  \ \nu=d(5g)^{5}.
\end{equation}
The only purpose of introducing $\nu$ is to simplify the exposition.
Let $h$ be the stable Arakelov height and let $h_{NT}$ be the N\'eron-Tate height. These heights are defined in (\ref{def:h}) and (\ref{def:nt}) respectively. 
Let $\mathcal C=\mathcal C(K)$ be the set of curves $X$ over $K$ of genus $\geq 2$ such that there is a finite morphism $X\to \mathbb P^1_K$ of prime degree which is geometrically a cyclic cover. 
Our first theorem establishes in particular Conjecture $(\cc)$ and  $\bb$ for all $X\in\mathcal C$. 
\begin{theorem}\label{thm:cyclic}
Let $X$ be a curve over $K$ of genus $g$, with set of bad reduction places $S$. If $X\in\mathcal C$, then there exist infinitely many $\p\in X(\bar{K})$ that satisfy
$$\log\max \bigl(h_{NT}(\p),h(\p)\bigl)\leq \nu^{d\nu}(N_SD_K)^\nu.$$
\end{theorem}

Next, we  state two results which improve Theorem \ref{thm:cyclic} in certain cases. We say that a curve $X$ over $K$ of genus $g$ is a hyperelliptic curve over $K$ if there is a finite morphism $X\to \mathbb P^1_K$ of degree two. For example, any genus two curve over $K$ is a hyperelliptic curve over $K$. We obtain the following theorem for hyperelliptic curves.

\begin{theorem}\label{thm:hyper}
Let $X$ be a hyperelliptic curve over $K$ of genus $g$, with set of bad reduction places $S$. If $\mathcal W$ denotes the set of Weierstrass points of $X$, then $$\sum_{\p\in \mathcal W}h_{NT}(\p)\leq \nu^{8^gd\nu}(N_SD_K)^\nu.$$
\end{theorem}
The N\'eron-Tate height $h_{NT}$ is non-negative, 
and any hyperelliptic curve over $K$ of genus $g$ has exactly $2g+2$ Weierstrass points. Thus Theorem \ref{thm:hyper} gives another proof of Conjecture $\bb$ for all hyperelliptic curves over $K$ of genus $g$. The next result refines  Theorem \ref{thm:cyclic} in the special case of genus two curves.

\begin{theorem}\label{thm:genus2}
Suppose $X$ is a curve over $K$ of genus two, with set of bad reduction places $S$. Then there exist infinitely many $\p\in X(\bar{K})$ that satisfy $$\max\bigl(h_{NT}(\p),h(\p)\bigl)\leq \nu^{2d\nu}(N_SD_K)^{\nu}, \ \ \ \nu=10^5d.$$
\end{theorem}

To state Proposition \ref{prop:abc} we need to recall the $abc$-conjecture  of Masser-Oesterl\'e \cite{masser:abc} over number fields. 
For any non-zero triple $\alpha,\beta,\gamma\in K$ we denote by $H(\alpha,\beta,\gamma)$ the usual absolute multiplicative Weil height  of the corresponding point in $\mathbb P^2(K)$, see  \cite[1.5.4]{bogu:diophantinegeometry}.
We define the height function $H_K=H^d$ and we write $S_K(\alpha,\beta,\gamma)=\prod N_v^{e_v}$ with the product extended over all finite places $v$ of $K$ such that $v(\alpha)$, $v(\beta)$ and $v(\gamma)$ are not all equal, where $e_v=v(p)$ for $p$ the residue characteristic of $v$. 
We mention that Masser \cite{masser:abc} added the ramification index $e_v$ in the definition of the support $S_K$ to obtain a natural behaviour of $S_K$ under finite field extensions. 

\vspace{0.3cm}
\noindent{\bf Conjecture ($abc$).} 
\emph{For any integer $n\geq 1$, and any real $r,\epsilon>1$, there is a constant $c$, which depends only on $n,r,\epsilon$, such that if $K$ is a number field of degree $[K:\QQ]\leq n$, and $\alpha,\beta,\gamma\in K$ are non-zero and satisfy $\alpha+\beta=\gamma$, then $H_K(\alpha,\beta,\gamma)\leq cS_K(\alpha,\beta,\gamma)^rD_K^\epsilon.$\vspace{0.3cm}}

The following proposition is conditional on Conjecture $(abc)$. It improves exponentially, in terms of $N_S$ and $D_K$, the inequalities in our theorems.
We put $u(g)=8^{(11g)^38^g}.$ 

\begin{proposition}\label{prop:abc}
Let $r,\epsilon>1$ be real numbers and write $$\Omega=(r+\frac{\epsilon}{d})\log N_S+\frac{\epsilon}{d}\log D_K.$$
Suppose Conjecture $(abc)$ holds for $n=24g^4d,r,\epsilon$ with the constant $c$.  Then there exist effective constants $c_1,c_2,c_3$, depending only on $c,r,\epsilon,d$ and $g$, such that for any curve $X$ over $K$ of genus $g$, with set of bad reduction places $S$, the following statements hold.
\begin{itemize}
\item[(i)] If $X\in\mathcal C$, then there are infinitely many points $\p\in X(\bar{K})$ that satisfy $$\log\max\bigl(h_{NT}(\p),h(\p)\bigl)\leq \nu^{\nu}\Omega+c_1.$$
\item[(ii)]   If  $X$ is a hyperelliptic curve over $K$, with set of Weierstrass points $\mathcal W$, then $$\sum_{\p\in \mathcal W}h_{NT}(\p)\leq u(g)(3g-1)(8g+4)\Omega+c_2.$$
\item[(iii)] If $g=2$, then there are infinitely many points $\p\in X(\bar{K})$ that satisfy $$\frac{1}{4}h_{NT}(\p)\leq h(\p)\leq  6u(2)\Omega+c_3.$$
\end{itemize}
\end{proposition}

We  note that the factor $u(g)$, which appears in Proposition \ref{prop:abc}, is not optimal. Its origin shall be explained after Proposition \ref{prop:cyclic}. Further, we point out that Proposition \ref{prop:abc} only requires the validity of $(abc)$ for some fixed $n,r,\epsilon$, instead of all  $n,r,\epsilon$, and $(abc)$ with fixed $n,r,\epsilon$ is often called ``weak $abc$ conjecture''. We also mention that  Elkies \cite{elkies:abcmordell} used  Belyi's theorem to show that  (effective) $(abc)$ implies (effective) Mordell. 
However, it is not clear if Conjecture ($\cc$) or Conjecture $\bb$  follows from the effective version of the Mordell conjecture which Elkies deduces from an effective $(abc)$.

In general, we conducted some effort to obtain constants reasonably close to the best that can be acquired with the present method of proof. However, to simplify the form of our inequalities we freely rounded off several of the numbers in our estimates.

\section{Self-intersection, Belyi degree and heights}\label{sec:curves}

In the first part of this section we give two lemmas which describe properties of the Belyi degree, and in the second part we collect results from Arakelov theory for arithmetic surfaces. We also prove a lemma which was used in Section \ref{sec:conj}.
Throughout this section we denote by $X$  a curve of genus $g\geq 2$, defined over a number field $K$.

If $L$, $\omega$ and $(\cdot ,\cdot)$ are as in Section \ref{sec:conj}, then  the stable self-intersection $e(X)$ of $\omega$ is the real number defined by
\begin{equation}\label{def:ex}
[L:\QQ]e(X)=(\omega,\omega).
\end{equation}
We observe that this definition does not depend on any choices. Let $D$ be the set of degrees of finite morphisms $X_{\bar{K}}\to \mathbb P^1_{\bar{K}}$ which are unramified outside $0,1,\infty$. Belyi's theorem \cite{belyi:theorem} shows that $D$ is non-empty. The Belyi degree $\deg_B(X)$ of $X$ is defined by
\begin{equation}\label{def:belyi}
 \deg_B(X)=\min D.
\end{equation}

Our proof of Theorem \ref{thm:cyclic} uses two fundamental properties of the Belyi degree. We now state the first of these properties in the following lemma.  

\begin{lemma}\label{lem:ex}
It holds
$e(X)\leq 10^8\deg_B(X)^5g$.
\end{lemma}
\begin{proof}The statement is proven in \cite[Theorem 1.1.1]{javanpeykar:belyi}. \end{proof}

The next lemma gives the second property. It is a consequence of an effective version  of Belyi's theorem \cite{belyi:theorem} worked out by Khadjavi  \cite{khadjavi:belyi}. We denote by $H(\alpha)$  the usual absolute multiplicative Weil height of $\alpha\in \mathbb P^1(\bar{K})$, defined in \cite[1.5.4]{bogu:diophantinegeometry}. Then we define the height $H_\Lambda$ of a subset $\Br$ of $\mathbb P^1(\bar{K})$ by 
$H_\Br=\sup\{H(\br),\br\in\Br\}$.  
\begin{lemma}\label{lem:belyi}
If $\varphi:X\to \mathbb P^1_K$ is a finite morphism, with set of (geometric) branch points $\Br\subset \mathbb P^1(\bar{K})$ and if $\uu=4[K:\QQ](\deg (\varphi)+g-1)^2$, then 
$$\deg_{B}(X) \leq (4\uu H_\Br)^{9\uu^3 2^{\uu-2}\uu!}\deg(\varphi).$$
\end{lemma}
\begin{proof} 
The absolute Galois group $\textnormal{Gal}(\bar{\QQ}/\QQ)$ of $\bar{\QQ}$ over $\QQ$ acts in the usual way on $\mathbb P^1(\bar{\QQ})\cong\mathbb P^1(\bar{K})$. Let $\Br'=\textnormal{Gal}(\bar{\QQ}/\QQ)\cdot \Br$ be the image of $\Br$ under this action. The classical result of Hurwitz \cite[Theorem 7.4.16]{liu:ag} implies that  $[K(\br):K]$ and $\lvert \Br\rvert$ are at most $2g-2+2\deg(\varphi)$ for $K(\lambda)$  the field of definition of $\br\in \Br$. This gives  $\lvert\Br'\rvert\leq \uu$, and the Galois invariance \cite[1.5.17]{bogu:diophantinegeometry} of $H$ shows  $H_\Br=H_{\Br'}$.
Then an application of \cite[Theorem 1.1]{khadjavi:belyi} with the Galois stable set $\Br'$ gives a finite morphism $\psi:\mathbb P^1_{\bar{K}}\to \mathbb P^1_{\bar{K}}$ with the following properties. The morphism $\psi$ is unramified outside $0,1,\infty$, with
$\psi(\Br)\subseteq \{0,1,\infty\}$ and
$$\deg(\psi)\leq (4\uu H_\Br)^{9\uu^3 2^{\uu-2}\uu !}.$$
We observe that the composition $\psi\circ\varphi: X_{\bar{K}}\to \mathbb P^1_{\bar{K}}$ is unramified outside $0,1,\infty$. This shows  $\deg_B(X)\leq \deg(\psi)\deg(\varphi)$ and then the displayed inequality implies Lemma \ref{lem:belyi}. 
\end{proof}

We mention that the bound in Khadjavi's \cite[Theorem 1.1]{khadjavi:belyi}, and thus Lemma \ref{lem:belyi}, can be improved in special cases. See for example Li{\c{t}}canu \cite{litcanu:belyi}.

Let $h$ be the stable Arakelov height from (\ref{def:h}). To obtain infinitely many small points we shall use the following lemma which relies on a fundamental result of Zhang \cite{zhang:positive}. 

\begin{lemma}\label{lem:zhang}
If $\epsilon>0$ is a real number, then there are infinitely many $\p\in X(\bar{K})$ with $$2(g-1)h(\p)\leq e(X)+\epsilon.$$
\end{lemma}
\begin{proof} It follows for example from Lemma \ref{lem:nth} (i) below that any $x\in X(\bar{K})$ satisfies $h(x)\geq 0$. Then we see that \cite[Theorem 6.3]{zhang:positive} implies the statement.\end{proof}

We remark that Moret-Bailly showed that for any real $\epsilon>0$, there exists an algebraic point $x\in X(\bar{K})$ which satisfies  $4(g-1)h(x)\leq e(X)+\epsilon$. For details we refer to the proof of \cite[Proposition 3.4]{moret-bailly:effmordell} which uses  \cite[Corollary, p.406]{faltings:arithmeticsurfaces} of Faltings.

The following lemma is a direct consequence of classical results of Arakelov \cite{arakelov:intersectiontheory}, Faltings \cite{faltings:arithmeticsurfaces} and Moret-Bailly \cite{moba:noether}. To state this lemma we introduce further notation. For any embedding $\sigma:K\hookrightarrow \CC$, we denote by $X_\sigma$ the compact connected Riemann surface which corresponds to the base change of $X$ to $\mathbb C$ with respect to $\sigma$. Let
$
\delta(X_\sigma)
$
be Faltings' delta invariant of $X_\sigma$, defined in \cite[p.402]{faltings:arithmeticsurfaces}. We denote by $M_g(\CC)$  the moduli space of smooth, projective and connected curves over $\CC$ of genus $g$. Faltings' delta invariant, viewed as a function $M_g(\CC)\to \mathbb R$, has a minimum which we denote by 
\begin{equation}\label{def:delta}
c_\delta(g).
\end{equation}
We mention that if $g\geq 3$, then effective lower bounds for $c_\delta(g)$  are not known. As above, we denote by $h_F(J)$ the absolute stable Faltings height of the Jacobian $J=\textnormal{Pic}^0(X)$ of $X$ and we let $h_{NT}$ be the N\'eron-Tate height defined in (\ref{def:nt}).

\begin{lemma}\label{lem:nth}
The following statements hold.
\begin{itemize}
\item[(i)] Any $\p\in X(\bar{K})$ satisfies $h_{NT}(\p)\leq 2g(g-1)h(\p).$
\item[(ii)] It holds $e(X)+c_\delta(g)\leq 12h_F(J)+4g\log(2\pi).$
\end{itemize}
\end{lemma}
\begin{proof}
To show statement (i) we take $\p\in X(\bar{K})$. Let $L$, $\mathcal X\to B$, $\omega$ and $(\cdot,\cdot)$ be as in (\ref{def:h}). We identify $\p$ with the corresponding Arakelov divisor on $\mathcal X$. Let $\Phi$ be the (unique) vertical $\QQ$-Cartier divisor on $\mathcal X$ such that the supports of $\Phi$ and $\p(B)$ are disjoint and any irreducible component $\Gamma$ of any fiber of $\mathcal X\to B$ satisfies $((2g-2)x-\omega+\Phi,\Gamma)=0$. 
Szpiro \cite[p.276]{szpiro:effective} observed that the Arakelov adjunction formula in \cite{arakelov:intersectiontheory} together with  \cite[Theorem 4.c)]{faltings:arithmeticsurfaces} leads to
$$2h_{NT}(\p)=-e(X)+4g(g-1)h(x)+\frac{1}{[L:\QQ]}(\Phi,\Phi).$$
Further, \cite[Theorem 5.a)]{faltings:arithmeticsurfaces} gives $-e(X)\leq 0$. Thus $(\Phi,\Phi)\leq 0$ implies (i).

We now prove (ii). If $v$ is a closed point  of $B$, then we denote by $\delta_v$ the number of singular points of the geometric special fiber of $\mathcal X$ over $v$. The (logarithmic) stable discriminant $\Delta(X)$ of $X$ is the real number defined by 
\begin{equation}\label{def:discriminant}
[L:\QQ]\Delta(X) =\sum\delta_v\log N_v
\end{equation}
with the sum taken over all closed points $v$ of $B$. Moret-Bailly's refinement \cite[Th\'eor\`eme 2.5]{moba:noether} of the Noether formula \cite[Theorem 6]{faltings:arithmeticsurfaces} implies
$$
12h_F(J)=\Delta(X)+e(X)-4g\log(2\pi)+\frac{1}{[L:\QQ]}\sum \delta(X_\sigma)
$$
with the sum taken over all embeddings $\sigma:L\hookrightarrow \mathbb C$. 
Then the estimates $\Delta(X)\geq 0$ and $\sum \delta(X_\sigma)
\geq [L:\QQ]c_\delta(g)$ prove (ii). This completes the proof of Lemma \ref{lem:nth}. 
\end{proof}

The final lemma of this section proves two results which were used in our discussion of Conjecture $(\cc)$ and $\bb$ in Section \ref{sec:conj}. Let $S$ be the set of finite places of $K$ where $X$ has bad reduction, and let $D_K$, $d$ and $N_S$ be the quantities defined in (\ref{def:para}).

\begin{lemma}\label{lem:conj}
Suppose $\epsilon>0$ is a real number. Then the following statements hold.
\begin{itemize}
\item[(i)] For any $x_0\in X(\bar{K})$, there are infinitely many $x\in X(\bar{K})$ that satisfy $$h_{NT}(x)\leq 4g^2(g-1)h(x_0)+\epsilon.$$
\item[(ii)] Suppose that Conjecture $(h)$ holds for $\epsilon$ and $n=dg$ with the constant $c$. Then there is an effective constant  $\kappa$, depending only on $c,\epsilon,d,g$, with the following properties. There exists $x_0\in X(\bar{K})$ that satisfies 
$$h(x_0)\leq \frac{3g}{g-1}\bigl(g+\frac{2\epsilon}{d}\bigl)(\log N_S+\log D_K)-\frac{c_\delta(g)}{4(g-1)}+\kappa$$
and there exist infinitely many points $x\in X(\bar{K})$ that satisfy $$h_{NT}(x)\leq 12g^2\bigl(g+\frac{2\epsilon}{d}\bigl)(\log N_S+\log D_K)-gc_\delta(g)+\kappa.$$
\end{itemize}
\end{lemma}

\begin{proof}
We first prove (i). On combining Lemma  \ref{lem:zhang} with Lemma \ref{lem:nth} (i), we obtain infinitely many points $x\in X(\bar{K})$ that satisfy 
\begin{equation}\label{eq:infpoints}
h_{NT}(x)\leq ge(X)+\epsilon.
\end{equation}
Furthermore, \cite[Theorem 5.b)]{faltings:arithmeticsurfaces} gives that any $x_0\in X(\bar{K})$ satisfies the inequality $e(X)\leq 4g(g-1)h(x_0)$. Hence we conclude assertion (i).

To show (ii) we may and do assume that Conjecture $(h)$ holds for $\epsilon$ and $n=dg$ with the constant $c$. The Weil restriction $A=\textnormal{Res}_{K/\QQ}(J)$ of the Jacobian $J$ of $X$ is an abelian variety over $\QQ$ of dimension $n$, which is geometrically isomorphic to $\prod J^\sigma$. Here the product is taken over all embeddings $\sigma:K\hookrightarrow \bar{K}$ and $J^\sigma$ is the base change of $J$ to $\bar{K}$ with respect to $\sigma$. 
Furthermore, the Galois invariance $h_F(J)=h_F(J^\sigma)$ implies that  $h_F(A)=dh_F(J)$. Thus an application of Conjecture $(h)$ with $A$ gives 
$$dh_F(J)\leq \bigl(\frac{n}{2}+\epsilon\bigl)\log N_A+c.$$
Let $N_J$ be the norm from $K$ to $\QQ$ of the usual conductor ideal of $J$ over $K$. 
Milne \cite[Proposition 1]{milne:arithmetic} showed that $N_A=N_JD_K^{2g}$, and Lockhart-Rosen-Silverman \cite[Theorem 0.1]{lorosi:conductor} implies that $N_J\leq \kappa_1N_S^{2g}$ for $\kappa_1$ an effective constant depending only on $d,g$. 
Then Lemma \ref{lem:nth} (ii) together with the displayed upper bound for $dh_F(J)$ leads to 
$$e(X)\leq 12g\bigl(g+\frac{2\epsilon}{d}\bigl)(\log N_S+\log D_K)-c_\delta(g)+\kappa_2,$$
where $\kappa_2$ is an effective constant depending only on $c,\epsilon,d$ and $g$. There exists an algebraic point $x_0\in X(\bar{K})$ with $4(g-1)h(x_0)\leq e(X)+\epsilon$, see the proof of \cite[Proposition 3.4]{moret-bailly:effmordell}. Therefore the  displayed upper bound for $e(X)$ together with (\ref{eq:infpoints}) implies assertion (ii). This completes the proof of Lemma \ref{lem:conj}.
\end{proof}

We conclude this section by the following remarks. Suppose there exists a finite morphism $\varphi: X\to \mathbb P^1_{K}$, with $\deg(\varphi)$ and $H_\Lambda$ effectively bounded in terms of $K$, $S$ and $g$, where $H_\Lambda$ is the height of the set $\Lambda$ of (geometric) branch points of $\varphi$. Then Lemma \ref{lem:ex}, Lemma \ref{lem:belyi} and Lemma \ref{lem:zhang} show that $X$ has infinitely many ``small points'', and therefore $X$ satisfies in particular Szpiro's small points conjecture. 

Similarly, if $\deg_B(X)$ is effectively bounded in terms of $K$, $S$ and $g$, then Lemma \ref{lem:ex} and Lemma \ref{lem:zhang} show that $X$ has infinitely many ``small points''. For example, if the base change of $X$ to $\CC$, with respect to some embedding $K\hookrightarrow \CC$, is a classical congruence modular curve, then a result of Zograf in \cite{zograf:rademacher} gives $\deg_B(X)\leq 128(g+1)$.

\section{Cyclic covers of prime degree}\label{sec:cyclic}

In this section we prove Theorem \ref{thm:cyclic} and Proposition \ref{prop:abc} (i).  We also give Proposition \ref{prop:cyclic} which may be of independent interest. It improves the inequalities in Theorem \ref{thm:cyclic} and Proposition \ref{prop:abc} (i), with the disadvantage inherent that it involves a constant which is not known to be effective.

Let $X$ be a curve over a number field $K$ of genus $g\geq 2$, let $S$ be the set of finite places of $K$ where $X$ has bad reduction and let $\mathcal C=\mathcal C(K)$ be the set of cyclic covers of prime degree from Section \ref{sec:results}. In this section we assume throughout that $X\in\mathcal C$.

We now give two lemmas which will be used in our proof of Theorem \ref{thm:cyclic}. 
To state and prove these lemmas we have to introduce some notation. 
Our assumption that $X\in \mathcal C$ provides a finite morphism $\varphi: X\to \mathbb P^1_K$ which is geometrically a cyclic cover of prime degree. 
Let $q$ be the degree of $\varphi$ and let $L$ be a finite extension  of $K$. 
We denote by $U=S(L,q)$ the set of places of $L$ which divide $q$ or a place in $S$. 
Let $\mathcal O_U^{\times}$ be the $U$-units in $L$ and let $h(\alpha)$ be the
 usual absolute logarithmic Weil height of $\alpha\in L$, defined in \cite[1.6.1]{bogu:diophantinegeometry}. 
We define the quantity $\mu_{U}=\sup(h(\lambda), \lambda \in \mathcal O_U^\times \textnormal{ and }  1-\lambda \in \mathcal O_U^{\times})$. 
Let $\mathcal R$ be the set of field extensions $L$ of $K$ such that $L$ is the
compositum of the fields of definition of four distinct (geometric) ramification points of $\varphi$. 
Then we define 
\begin{equation}\label{def:ux}
\mu_X=\sup(1,\mu_{S(L,q)})
\end{equation}
with the supremum taken over all fields $L\in \mathcal R$. Let $\deg_B(X)$ be the Belyi degree of $X$, defined in (\ref{def:belyi}), and let $d$ be the degree of $K$ over $\QQ$.
\begin{lemma}\label{lem:degbux}
If $\nu=d(5g)^5$, then $\log\deg_{B}(X) \leq \nu^{\nu/2}\mu_X$.
\end{lemma}
To prove this lemma we shall combine the estimate for $\deg_B(X)$ in Lemma \ref{lem:belyi} with the following observation of Par{\v{s}}in in \cite{parshin:shafarevich}: 
The cross-ratios of the  branch points, of a hyperelliptic map of a genus two curve $X$ over $K$, are solutions of certain $S(L,2)$-unit equations.
Oort \cite[Lemma 2.1]{oort:shafarevich} and de Jong-R\'emond \cite[Proposition 2.1]{jore:shafarevich} generalized Par{\v{s}}in's idea to hyperelliptic curves  and to cyclic covers of prime degree.
\begin{proof}[Proof of Lemma \ref{lem:degbux}] 
We take three distinct (geometric) ramification points of $\varphi$. Let $M$ be the compositum of their fields of definition. On composing $\varphi$ with a suitable automorphism of  $\mathbb P^1_{M}$,
we get a finite morphism $X_{M}\to \mathbb P^1_{M}$ of degree $q$ such that $\{0,1,\infty\}\subset\Br$, where $\Br$ is the set of (geometric) branch points of $\varphi$. Let $H_\Lambda$ be the height of $\Br$, defined in  Section \ref{sec:curves}. To prove the inequality
$$
\log H_\Br\leq \mu_X,
$$
we may and do take $\br\in\Br$ with $\lambda\neq 0,1,\infty$. We write $U=S(K(\lambda),q)$. 
From \cite[Proposition 2.1]{jore:shafarevich} we deduce that the cross-ratios $\br=\textnormal{cr}(\infty,0,1,\br)$ and $1-\lambda=\textnormal{cr}(\infty,1,0,\br)$  are $U$-units in $K(\lambda)$. 
This implies that $h(\br)\leq \mu_X$, since $K(\lambda)\subseteq L$ for some $L\in\mathcal R$. Hence we obtain $\log H_\Br\leq \mu_X$ as desired.
Next, we observe that the ramification indexes of  $\varphi_{\bar{K}}:X_{\bk}\to \mathbb P^1_{\bk}$  are in $\{1,q\}$, and  $\textnormal{Gal}(\bk/K)$ acts on the (geometric) ramification points of $\varphi$.
Therefore  Hurwitz leads to $q\leq 2g+1$ and $[M:\QQ]\leq 15dg^3$. Then an application of Lemma \ref{lem:belyi} with $X_{M}\to \mathbb P^1_{M}$ gives an upper bound for $\deg_B(X)$ which together with the displayed inequality implies Lemma \ref{lem:degbux}.
\end{proof}

We remark that if $L\in\mathcal R$, then  Hurwitz (see the proof of Lemma \ref{lem:degbux}) leads to 
$q\leq 2g+1$ and $[L:K]\leq 24g^4,$
and \cite[Lemme 2.1]{jore:shafarevich} of de Jong-R\'emond implies that $L$ is unramified outside $S(K,q)$.

Next, we go into number theory and we give an upper bound for $\mu_X$ in terms of the  quantities $N_S$, $\nu$, $d$ and $D_K$ which are defined in (\ref{def:para}).
\begin{lemma}\label{lem:ux}
The following statements hold.
\begin{itemize}
\item[(i)] It holds $\mu_X\leq \nu^{d\nu/8}(N_SD_K)^{\nu}$.
\item[(ii)]  Let $r,\epsilon>1$ be real numbers. Suppose $(abc)$ holds for $n=24g^4d,r,\epsilon$ with the constant $c$. Then there exists an effective constant $c^*$, depending only on $c,r,\epsilon,d$ and $g$, such that $d\mu_{X}\leq  (dr+\epsilon)\log N_S+\epsilon\log D_K+c^*.$
\end{itemize}
\end{lemma}  

To prove (i) we use \cite[Proposition 6.1 (ii)]{rvk:szpiro}. It is based on the theory of logarithmic forms and we refer to the monograph of Baker-W\"ustholz \cite{bawu:logarithmicforms} in which the state of the art of this theory is exposed. 

\begin{proof}[Proof of Lemma \ref{lem:ux}]
We take $L\in \mathcal R$, and we write $U=S(L,q)$, $T=S(K,q)$ and $l=[L:K]$. Then we observe that $N_T=\prod_{v\in T}N_v$ satisfies $N_T\leq q^dN_S$. Furthermore, the remark after the proof of Lemma \ref{lem:degbux} gives that $q\leq 2g+1$ and $l\leq 24g^4,$ and that $L$ is unramified outside $T$. 
We now apply \cite[Proposition 6.1]{rvk:szpiro} with  $U=U$, $T=T$ and $S=T$, where the symbols $U,T,S$ on the left hand side of these equalities denote the sets in \cite[Proposition 6.1]{rvk:szpiro}. In particular, \cite[Proposition 6.1 (ii)]{rvk:szpiro} leads to 
$
\mu_U\leq \nu^{d\nu/8}(D_KN_S)^\nu
$ 
which implies statement (i).

To show statement (ii) we take real numbers $r,\epsilon>1$. We may and do assume that Conjecture $(abc)$ holds for $n=24g^4d,r,\epsilon$ with the constant $c$. Then Conjecture $(abc)$ holds in particular for $n'=ld,r,\epsilon$ with the same constant $c$, since $n'\leq n$.  Therefore we see that \cite[Proposition 6.1 (iii)]{rvk:szpiro} gives
$$d\mu_U\leq (dr+\epsilon)\log N_T+\epsilon\log D_K+\epsilon d t\log l-\frac{\epsilon}{l}\log N_T+\frac{1}{l}\log c$$  
for $t=\lvert T\rvert$. From \cite[Lemma 6.3]{rvk:szpiro} we get that  $\epsilon d t\log l-\frac{\epsilon}{l}\log N_T$ is bounded from above by an effective constant, which depends only on  $\epsilon,d$ and $g$. Hence we deduce statement (ii) and this completes the proof of Lemma \ref{lem:ux}.
\end{proof}

We recall that $h$ denotes the stable Arakelov height and that $h_{NT}$ denotes the N\'eron-Tate height. These heights are defined in (\ref{def:h}) and (\ref{def:nt}) respectively. We now prove Theorem \ref{thm:cyclic} and Proposition \ref{prop:abc} (i) simultaneously.

\begin{proof}[Proof of Theorem \ref{thm:cyclic} and Proposition \ref{prop:abc} (i)] 
On combining  Lemma \ref{lem:ex}, Lemma \ref{lem:zhang}, Lemma \ref{lem:nth} and Lemma \ref{lem:degbux}, we obtain
infinitely many $\p\in X(\bar{K})$ that satisfy
$\log \max\bigl(h_{NT}(\p),h(\p)\bigl)\leq 6\nu^{\nu/2}\mu_X.$
Then Lemma \ref{lem:ux} (i) and Lemma \ref{lem:ux} (ii)  imply Theorem \ref{thm:cyclic} (i) and  Proposition \ref{prop:abc} (i) respectively.
\end{proof}

The remaining of this section is devoted to the following Proposition \ref{prop:cyclic}. Let $c_\delta(g)$ be the minimum of Faltings' delta invariant on $M_g(\CC)$, defined in (\ref{def:delta}). We recall that if $g\geq 3$, then effective lower bounds for $c_\delta(g)$ in terms of $g$ are not known. We write $u(g)=8^{(11g)^38^g}$ and now we can state the following result.

\begin{proposition}\label{prop:cyclic}
The following statements hold. 
\begin{itemize}
\item[(i)] There are infinitely many $\p\in X(\bar{K})$ that satisfy
$$h(\p)\leq \nu^{8^gd\nu}(D_KN_S)^\nu-c_\delta(g).$$
\item[(ii)] Let $r,\epsilon>1$ be real numbers. Suppose Conjecture $(abc)$ holds for $n=24g^4d,r,\epsilon$ with the constant $c$. Then there exists an effective constant $c_1'$, depending only on $c,r,\epsilon,d$ and $g$, with the property that there are infinitely many points $\p\in X(\bar{K})$ which satisfy $$h(\p)\leq 6\frac{u(g)}{g-1}\bigl((r+\frac{\epsilon}{d})\log N_S+\frac{\epsilon}{d}\log D_K\bigl)+c_1'-\frac{c_\delta(g)}{2g-2}.$$
\end{itemize}
\end{proposition}

These (ineffective) inequalities improve exponentially, in terms of $N_S$ and $D_K$, the estimates provided by Theorem \ref{thm:cyclic} and Proposition \ref{prop:abc} (i).  On using Lemma \ref{lem:nth} one can formulate Proposition \ref{prop:cyclic} also in terms of $h_{NT}$. We mention that the factor $u(g)$ comes from explicit height comparisons of R\'emond \cite{remond:rational} which rely inter alia on  results of Bost-David-Pazuki \cite{pazuki:heights}. R\'emond's explicit height comparisons hold for arbitrary curves over $K$ and in our special case $X\in\mathcal C$ it seems possible to improve these height comparisons, and thus $u(g)$, up to a certain extent.

To prove Proposition \ref{prop:cyclic} we use the following tools. 
We combine Lemma \ref{lem:ux} with \cite[Proposition 4.1 (i)]{rvk:szpiro}. 
This slightly refines the method developed by Par{\v{s}}in \cite{parshin:shafarevich}, Oort \cite{oort:shafarevich} and de Jong-R\'emond \cite{jore:shafarevich}.
We also use Lemma \ref{lem:zhang} which is based on a theorem of Zhang \cite{zhang:positive}, and Lemma \ref{lem:nth} (ii) which relies on Moret-Bailly's refinement \cite{moba:noether}  of the Noether formula in Faltings' article \cite{faltings:arithmeticsurfaces}.

\begin{proof}[Proof of Proposition \ref{prop:cyclic}]
Let $h_F(J)$ be the absolute stable Faltings height of the Jacobian $J$ of $X$. The remark below \cite[Proposition 4 (i)]{rvk:szpiro} gives 
\begin{equation}\label{eq:fhux}
h_F(J)\leq u(g)\mu_X.
\end{equation}
Then Lemma \ref{lem:nth} (ii) shows $e(X)\leq 12u(g)\mu_X-c_\delta(g)+4g\log(2\pi)$ for $e(X)$  as in (\ref{def:ex}). Hence  Lemma \ref{lem:zhang} and Lemma \ref{lem:ux} imply Proposition \ref{prop:cyclic}.
\end{proof}

We conclude this section by the following remarks. Let $h_F(J)$ be as above, let $\Delta(X)$ be the stable discriminant of $X$ defined in (\ref{def:discriminant}) and let $e(X)$ be as in (\ref{def:ex}). Further, we define the quantity $\delta(X)=\frac{1}{d}\sum\delta(X_\sigma)$ with the sum taken over all embeddings $\sigma:K\hookrightarrow \CC$, where  
 $\delta(X_\sigma)$ is defined in Section \ref{sec:curves}. Then it holds 
\begin{equation}\label{eq:general}
\log\max(e(X),\delta(X),h_F(J),\Delta(X))\leq \nu^{d\nu}(D_KN_S)^\nu.
\end{equation} 
Indeed \cite[Theorem 1.1.1]{javanpeykar:belyi} gives that $\max(e(X),\delta(X),h_F(J),\Delta(X))$ is at most $10^9g^2\deg_B(X)^5$ and then Lemma \ref{lem:degbux} and Lemma \ref{lem:ux} prove the displayed inequality.

Next, we mention that de Jong-R\'emond  obtained in \cite[Theorem 1.2]{jore:shafarevich} a better estimate for $h_F(J)$ and that \cite[Theorem 3.2]{rvk:szpiro} gives an exponential better upper bound for $\Delta(X)$. The latter bound involves a constant, depending at most on $g$, which is only known to be effective for hyperelliptic curves $X$ over $K$. In addition, we note that \cite[Theorem 1.2]{jore:shafarevich} and \cite[Theorem 3.2]{rvk:szpiro}  both depend inter alia on the above mentioned explicit height comparisons of R\'emond in \cite{remond:rational}, and such height comparisons are not used in our proof of the displayed inequality.

We point out that our method using the Belyi degree gives that Theorem \ref{thm:cyclic}, Proposition \ref{prop:abc} (i) and inequality (\ref{eq:general}) hold more generally for any curve $Y$ over $K$ which admits a finite \'etale morphism to some $X\in \mathcal C$. Indeed on using that $Y$ is an \'etale cover of $X$ we deduce from Hurwitz that $\deg_B(Y)$ is explicitly bounded in terms of $\deg_B(X)$ and the genus of $Y$, and then the above arguments prove the generalization. Here we used in addition \cite[Corollary 4.10]{lilo:modelsofcurves} which assures that if $v$ is an arbitrary finite place of $K$ and $X$ has bad reduction at $v$, then $Y$ has bad reduction at $v$.

Finally, we mention that  the method, used in our paper, gives an affine plane model $f(z_1,z_2)=0$ of $X_{\bar{K}}$ such that if $a_i\in\bar{K}$ is a coefficient of $f$, then $h(a_i)\leq \nu^{d\nu}(D_KN_S)^\nu.$

\section{Hyperelliptic curves}\label{sec:hyper}

In this section we prove Theorem \ref{thm:hyper}. We also show two lemmas which provide explicit results for certain (Arakelov) invariants of hyperelliptic curves. Throughout this section we denote by $X$ a hyperelliptic curve over a number field $K$ of genus $g\geq 2$.

As before, we denote by $X_\sigma$ the compact connected Riemann surface corresponding to the base change of $X$ to $\mathbb C$ with respect to an embedding $\sigma:K\hookrightarrow \mathbb C$. Let $T(X_\sigma)$ be the invariant of de Jong. It is the norm of a canonical isomorphism between certain line bundles on $X_\sigma$ and we refer to \cite[Definition 4.2]{dejong:riemanninvariants} for a precise definition of $T(X_\sigma)$.

\begin{lemma}\label{lem:tx}
It holds $-36g^3\leq \log T(X_\sigma)$.
\end{lemma}

To prove this lemma we use de Jong's formula  \cite[Theorem 4.7]{dejong:riemanninvariants}. It expresses $T(X_\sigma)$ in terms of a certain hyperelliptic discriminant modular form, which we then estimate on using the explicit inequality in \cite[Lemma 5.4]{rvk:szpiro}.

\begin{proof}[Proof of Lemma \ref{lem:tx}]

We begin to state the formula for $T(X_\sigma)$ in \cite[Theorem 4.7]{dejong:riemanninvariants}.
Let $\mathfrak H_g$ be the Siegel upper half plane of complex symmetric $g\times g$ matrices with positive definite imaginary part. 
We denote by $\Delta_g$ the hyperelliptic discriminant modular form on $\mathfrak H_g$, defined in \cite[Section 5]{rvk:szpiro}. 
Since $X$ is hyperelliptic, there exists a 
finite morphism $\varphi:X_\sigma\to \mathbb P_{\mathbb C}^1$ of degree two. 
Write $H_1(X_\sigma,\ZZ)$ for the first 
homology group of $X_\sigma$ with coefficients in $\ZZ$. 
As in Mumford \cite[Chapter IIIa]{mumford:theta2}, 
we obtain a canonical symplectic basis of  $H_1(X_\sigma,\ZZ)$ with respect to a 
fixed ordering of the $2g+2$ branch points of $\varphi$. 
Then \cite[Theorem 4.7]{dejong:riemanninvariants} delivers a 
basis of the global sections of the sheaf of differentials on $X_\sigma$ which has the following property:
Integration of this basis over
the canonical symplectic basis of $H_1(X_\sigma,\ZZ)$ gives 
a period matrix $\tau_\sigma \in\mathfrak H_g$ that satisfies  
\begin{equation*}
T(X_\sigma)=(2\pi)^{-2g}\left|\Delta_g(\tau_\sigma)\textnormal{det}(\textnormal{im} (\tau_\sigma))^{2a}\right|^{-(3g-1)/(8bg)},
\end{equation*}
where $a=\binom{2g+1}{g+1}$ and $b=\binom{2g}{g+1}$. Furthermore, \cite[Lemma 5.4]{rvk:szpiro} gives an effective constant $k_1$, depending only on $g$, such that $\left|\Delta_g(\tau_\sigma)\textnormal{det}(\textnormal{im} (\tau_\sigma))^{2a}\right|\leq k_1.$
The effective constant $k_1$ is  explicitly computed in \cite[(15)]{rvk:szpiro}, and then the displayed formula for $T(X_\sigma)$ leads to an inequality as stated in Lemma \ref{lem:tx}.
\end{proof}

Let $h_F(J)$ be  the absolute stable Faltings height of the Jacobian $J$ of $X$ and let $h_{NT}$ be the N\'eron-Tate height which is defined in (\ref{def:nt}).

\begin{lemma}\label{lem:wpf}
 If $\mathcal W$ denotes the set of Weierstrass points of $X$, then $$\sum_{\p\in \mathcal W}h_{NT}(\p)\leq (3g-1)(8g+4)h_F(J)+293g^5.$$
\end{lemma}
To prove  Lemma \ref{lem:wpf} we use de Jong's formula \cite[Theorem 4.3]{dejong:weierstrasspoints}. This formula involves inter alia $h_F(J)$ and $\sum_{\p\in \mathcal W}h_{NT}(\p)$, and an analytic term, related to $T(X_\sigma)$, which we control by Lemma \ref{lem:tx}.
\begin{proof}[Proof of Lemma \ref{lem:wpf}]
To state the formula in \cite[Theorem 4.3]{dejong:weierstrasspoints} we introduce quantities $A_1,A_2$ and $A_3$. 
Let $L$ be a finite extension  of $K$ such that $X_L$ has semi-stable reduction over the spectrum $B$ of the ring of integers of $L$ and such that all Weierstrass points of $X$ are in $X(L)$. 
We define
$$A_1=-4g(2g-1)(g+1)\log(2\pi)+\frac{8g^2}{[L:\QQ]}\sum\log T(X_\sigma)$$
with the sum taken over all embeddings $\sigma:L\hookrightarrow \mathbb C.$ Let $\mathcal X\to B$, $(\cdot,\cdot)$  and  $\omega$ be as in (\ref{def:h}). 
We denote by  $E$  the residual divisor on $\mathcal X$ defined in \cite[p.286]{dejong:weierstrasspoints}. Let $\Delta(X)$ be  the stable discriminant of $X$ defined in (\ref{def:discriminant}).
We take $$A_2=(2g-1)(g+1)\Delta(X)+\frac{4}{[L:\QQ]}(E,\omega).$$
For any section $\p\in \mathcal X(B)$ of $\mathcal X\to B$,  let $\Phi_\p$ be the (unique) vertical $\QQ$-Cartier divisor on $\mathcal X$ such that the supports of $\Phi_\p$ and $\p(B)$ are disjoint and  any irreducible component $\Gamma$ of any fiber of $\mathcal X\to B$ satisfies $((2g-2)x-\omega+\Phi_\p,\Gamma)=0$.
We write $$A_3=\frac{1}{[L:\QQ]g(g-1)}\sum_{\p\in \mathcal W}-(\Phi_\p,\Phi_\p)n_\p$$
with $n_\p$ the multiplicity of $\p$ in the Weierstrass divisor $W$   on $\mathcal X$, where $W$  is defined in \cite[p.286]{dejong:weierstrasspoints}. 
Then \cite[Theorem 4.3]{dejong:weierstrasspoints} gives
$$(3g-1)(8g+4)h_F(J)=A_1+A_2+A_3+\frac{2}{g(g-1)}\sum_{\p\in \mathcal W} h_{NT}(\p)n_\p.$$
We now estimate the quantities $A_1,A_2$ and $A_3$ from below. To deal with $A_1$ we use Lemma \ref{lem:tx}. It gives 
$-293g^5\leq A_1.$ 
The divisor $E$ on $\mathcal X$ is vertical and effective, and our minimal  $\mathcal X$ does not contain any exceptional curves. This implies that $(E,\omega)\geq 0$. Then $\Delta(X)\geq 0$ and $-(\phi_\p,\phi_\p)\geq 0$ show that $A_2$ and $A_3$ are both non-negative. Furthermore,  \cite[Lemma 3.2]{dejong:weierstrasspoints}  gives $n_x=g(g-1)/2$. Thus the above displayed formula and the lower bounds for $A_1,A_2$ and $A_3$ imply an inequality as claimed.
\end{proof}

On using the inequality given in Lemma \ref{lem:wpf} we now prove Theorem \ref{thm:hyper} and Proposition \ref{prop:abc} (ii) simultaneously.

\begin{proof}[Proof of Theorem \ref{thm:hyper} and Proposition \ref{prop:abc} (ii)]
Since $X$ is a hyperelliptic curve over $K$, there exists a finite morphism $X\to \mathbb P^1_K$  which is geometrically a cyclic cover of prime degree two. Then, on combining Lemma \ref{lem:wpf}, Lemma \ref{lem:ux} and inequality (\ref{eq:fhux}), we deduce Theorem \ref{thm:hyper} and Proposition \ref{prop:abc} (ii).
\end{proof}

 Let $x\in \mathcal W$. We remark that the arguments of Burnol \cite[Theorem B]{burnol:wp} imply an upper bound for $h_{NT}(x)$ in terms of certain Arakelov invariants of $X$.  However, it turns out that the bound for $h_{NT}(x)$ in Lemma \ref{lem:wpf} leads to a better inequality in Theorem \ref{thm:hyper}.

\section{Genus two curves}\label{sec:genus2}

In this section we prove Theorem \ref{thm:genus2} and Proposition \ref{prop:abc} (iii). 

Let $c_\delta(2)$ be the minimum of Faltings' delta invariant on $M_2(\CC)$, see (\ref{def:delta}). The following lemma was established in \cite[Proposition 5.1 (v)]{rvk:szpiro}.  

\begin{lemma}\label{lem:dx}
It holds $-186\leq c_\delta(2)$.
\end{lemma}

We now combine Lemma \ref{lem:dx} with the arguments used in the proof of Proposition \ref{prop:cyclic} to prove Theorem \ref{thm:genus2} and Proposition \ref{prop:abc} (iii).
\begin{proof}[Proof of Theorem \ref{thm:genus2} and Proposition \ref{prop:abc} (iii)]
Let $X$ be a genus two curve  over a number field $K$. It is a hyperelliptic curve over $K$ by  \cite[Proposition 7.4.9]{liu:ag}. Thus there exists a finite morphism $X\to \mathbb P^1_K$ which is geometrically a cyclic cover of prime degree two. As before, we denote by $h$ and $h_{NT}$ the stable Arakelov height and the N\'eron-Tate height respectively, defined in (\ref{def:h}) and (\ref{def:nt}). Then  Lemma \ref{lem:zhang}, Lemma \ref{lem:nth}, inequality (\ref{eq:fhux}) and Lemma \ref{lem:dx} give infinitely many $x\in X(\bar{K})$ that satisfy 
$$\frac{1}{4}h_{NT}(\p)\leq h(\p)\leq 6u(2)\mu_X+101$$ for $\mu_X$ as in (\ref{def:ux}) and $u(2)$ as in Proposition \ref{prop:abc}. Thus the upper bounds for $\mu_X$ in Lemma \ref{lem:ux} (i) and Lemma \ref{lem:ux} (ii)  imply Theorem \ref{thm:genus2} and Proposition \ref{prop:abc} (iii) respectively. 
\end{proof}

{\bf Acknowledgements:}  
Most of the results were obtained at the IAS Princeton. Both authors would like to thank the IAS for the friendly environment and for the generous financial support. 
In addition, R.v.K. would like to thank Professor Szpiro for giving helpful explanations. R.v.K. was supported 
by the
National Science Foundation under agreement No. DMS-0635607.

{\footnotesize
\bibliographystyle{amsalpha}
\bibliography{../../literature}
}

\vspace{0.5cm}

\noindent Mathematical Institute, University of Leiden,  2300 RA Leiden, The Netherlands\\
E-mail adress: {\sf ajavanp@math.leidenuniv.nl}

\vspace{0.5cm}

\noindent IH\'ES, 35 Route de Chartres, 91440 Bures-sur-Yvette, France\\
E-mail adress: {\sf rvk@ihes.fr}

\end{document}